\documentclass[11pt]{amsart}

\usepackage{xypic}  
\usepackage{amssymb}
\usepackage{amsthm}
\usepackage{epsfig}
\usepackage{paralist, comment}
\usepackage{lmodern}
\usepackage{longtable}
\usepackage{graphicx}

\usepackage{comment}

\usepackage[all,cmtip]{xy}
\usepackage{amsmath}
\usepackage{xcolor}

\newtheorem{theorem}{Theorem}

\newtheorem{thm}{Theorem}[section] 

\newtheorem{lemma}[thm]{Lemma}

\theoremstyle{definition}

  \newtheorem{definition-remark}[thm]{Definition-Remark}

\def\geq{\geqslant}

\def\c1{\operatorname{c_1}}
\def\c2{\operatorname{c_2}}

\def\Sym{\operatorname{Sym}}

\def\ZZ{{\mathbb Z}}

\def\PP{{\mathbb P}}
\def\KK{{\mathbb K}}

\def\C{{\mathcal C}}

\def\c{\mathfrak{c}}

\def\cong{\simeq}
\def\geq{\geqslant}

\def\+{\oplus}               
\def\*{\otimes}                  

\def\Num{\operatorname{Num}}

\def\Sym{\operatorname{Sym}}

\begin{document}

\title{On two extensions of Poncelet theorem}

\author[C.~Ciliberto]{Ciro Ciliberto}
\address{Ciro Ciliberto, Dipartimento di Matematica, Universit\`a di Roma Tor Vergata, Via della Ricerca Scientifica, 00173 Roma, Italy}
\email{cilibert@mat.uniroma2.it}



 \begin{abstract}  In this note we provide two extensions of a particular case of Poncelet theorem.
\end{abstract}

\maketitle


\vspace{-1cm}

\section{Introduction}\label{sec:intro}

One of the most beautiful results in classical projective geometry is Poncelet theorem concerning closed polygons which are inscribed in one conic and circumscribed about another. The first instance of this theorem is the following statement: if there is a triangle which is inscribed in a conic  and circumscribed about another conic  then there is a 1--dimensional family of such triangles. In this paper we give two extensions of this statement. The first one (see \S \ref {sec:2}) is a higher dimensional extension to pyramids inscribed in and circumscribed about rational normal curves.  Also in this case, if there is such a pyramid, there is a $1$--dimensional family of them, provided the two rational normal curves enjoy a certain condition with respect to the pyramid. This extension has been inspired  by a paper by E. Beltrami (see \cite{Beltr}) where he is concerned with the case of conics and of rational normal cubics, though it seems to us that his treatment is not fully correct. There is another classical extension of Poncelet theorem, due to G. Darboux and C. Segre (see \cite {Darb2, Darb,Seg}), treated in \S \ref {sec:3}, which deals with multilaterals that are circumscribed about a conic and inscribed in a curve: under suitable conditions, again, if there is one such multilateral, there is a $1$--dimensional family of them. We give here a short new proof of this fact, which paves the road for the further extension we  present in \S \ref {sec:4}. This has to do with a similar situation which occurs on the symmetric product of a curve with general moduli (see Theorem \ref {thm:ponc} for the precise statement). 

To finish, it should be mentioned that there are various other extensions of Poncelet theorem (see for example \cite {GRHA}).

\medskip

\noindent {\bf Aknowledgements:} The author is a member of GNSAGA of INdAM. He acknowledges the MIUR Excellence Department Project awarded to the Department of Mathematics, University of Rome Tor Vergata, CUP E83C18000100006.

\section{Inscribed and circumscribed pyramids to rational normal curves}\label {sec:2}

Let us consider a rational normal curve $C$ in $\PP^n$ over an algebraically closed field $\mathbb K$. Consider also an \emph{elementary pyramid} $\Delta$ with \emph{vertices} at $n+1$ independent points of $\PP^n$. The \emph{faces} of $\Delta$ are the hyperplanes spanned by all vertices but one. An elementary pyramid is uniquely determined by its $n+1$ vertices, as well as by its $n+1$  faces. All elementary pyramids are projectively equivalent to each other. 

We will say that $C$ is \emph{circumscribed} about the elementary pyramid $\Delta$ (or that $\Delta$ is \emph{inscribed} in $C$) if $C$ contains the vertices of $\Delta$. In this case the intersections of each face of $\Delta$ with $C$ are just the $n$ vertices of $\Delta$ on that face. 

We will say that $C$ is \emph{inscribed} in the elementary pyramid $\Delta$ (or that $\Delta$ is \emph{circumscribed} about $C$)  if each face of $\Delta$ \emph{hyperosculates} $C$ in one point, i.e., it has intersection multiplicity $n$ with $C$ in a point of the face.  In this case the unique intersection point of each face with $C$  lies on only that face and on no other face. 

Suppose we have an elementary pyramid $\Delta$ and a rational normal curve $C_1$ circumscribed to $\Delta$ and a rational normal curve $C_2$ inscribed in $\Delta$.  We will say that $C_1$ and $C_2$ are \emph{conjugate} with respect to $\Delta$ if there is a projectivity from $C_1\cong \PP^1$ to $C_2\cong \PP^1$ which maps each of the $n+1$ vertices of $\Delta$ on $C_1$ to the contact point of $C_2$ with the opposite face of $\Delta$. Note that in the case $n=2$ of the classical Poncelet theorem, this is an empty conditions, since the projectivity in question certainly exist. 

The first extension of Poncelet theorem we deal with is the following:

\begin{theorem}\label{thm:ponc1} Suppose there is an elementary pyramid $\Delta$ in $\PP^n$ and two rational normal curves $C_1, C_2$, the former circumscribed about $\Delta$, the latter inscribed in $\Delta$, conjugate with respect to $\Delta$. Then there are infinitely many pyramids inscribed in $C_1$ and circumscribed about $C_2$. More precisely, each point of $C_1$ [resp. of $C_2$] is the vertex [resp. is the contact point with $C_2$ of a face] of a unique elementary pyramid  inscribed in $C_1$ and circumscribed about $C_2$.
\end{theorem} 

\begin{proof} We may assume that $\Delta$ is the \emph{fundamental pyramid}, i.e., its faces are the hyperplanes $x_i=0$, for $i=0,\ldots, n$. If $[u_0,\ldots, u_n]$ are the dual coordinates of $[x_0,\ldots, x_n]$ in $\check{\PP}^n$, the opposite vertex to the face $x_i=0$, as a hyperplane of  $\check{\PP}^n$, has dual equation $u_i=0$, for $i=0,\ldots, n$.

We can consider the 1--dimensional family of hyperosculating hyperplanes to $C_2$, which can be represented by an equation of the form
\[
\sum_{i=0}^n \frac {A_ix_i} {t-a_i}=0
\]
where $t\in \KK$ is the parameter on which the hyperplanes in question depend.  Note that the hyperplane $x_i=0$ corresponds to the value $t=a_i$, hence $a_1,\ldots, a_n$ are all distinct. 

We can consider the variable point on $C_1$ in the dual coordinates $[u_0,\ldots, u_n]$. By the hypothesis that $C_1$ and $C_2$ are conjugate, the equation is of the form
\[
\sum_{i=0}^n \frac {B_iu_i} {s-a_i}=0
\]
where $s\in \KK$ is the parameter on which the point in question depends. Notice that by varying the constants $a_0, \ldots, a_n$, $A_0, \ldots, A_n$ and $B_0, \ldots, B_n$ we can assume that the curves $C_1$ and $C_2$ are general among those circumscribed about $\Delta$,  inscribed in $\Delta$ and conjugate. 

The condition for the osculating hyperplane to $C_2$ corresponding to the parameter $t$ to pass through the point of $C_1$ corresponding to the parameter $s$, is given by
\[
\sum_{i=0}^n \frac {A_iB_i} {(t-a_i)(s-a_i)}=0
\]
and, as one may expect, it is of degree $n$ in both variables $t$ and $s$. 

Consider now the equation
\begin{equation}\label{eq:belt}
\sum_{i=0}^n \frac {A_iB_i} {x-a_i}=\frac 1k
\end{equation}
where $k\in \KK$ is variable. Getting rid of the denominators, we see that this is an equation of degree $n+1$ in $x$, and we denote by $k_0,\ldots, k_n$ its roots, which depend on $k$. Note that for general values of $k$, the equation \eqref {eq:belt} has distinct roots, because this is the case when $k=0$, in that case the roots being $a_0, \ldots, a_n$ which are distinct. Let  $\ell,m$ be two distinct integers in $\{0,\ldots, n\}$, and consider two distinct roots  $k_\ell\neq k_m$ of \eqref {eq:belt}. Plugging both values in \eqref {eq:belt} and equating the left hand sides, one gets the  ${{n+1}\choose 2}$ relations 
\[
\sum_{i=0}^n \frac {A_iB_i} {(k_\ell-a_i)(k_m-a_i)}=0.
\]
This means that if we set 
\[
t=s=k_i, \quad i=0,\ldots, n
\]
we find a pyramid with faces corresponding to the values $t=k_0,\ldots, k_n$ and vertices corresponding to values $s=k_0,\ldots, k_n$, which are both circumscribed to $C_2$ and inscribed to $C_1$. By varying $k$, also this pyramid varies and since each of the roots $k_0,\ldots, k_n$ can assume all possible values, then the vertices of the pyramid on $C_1$ [resp. the contact points of the faces with $C_2$] vary describing all of $C_1$ [resp. all of $C_2$]. \end{proof}

\section{Multilaterals circumscribed about conics and inscribed in curves}\label {sec:3}

Consider in $\PP^2$ a \emph{general $(n+1)$--lateral} $L$, i.e., the closed set formed by $n+1$ distinct lines, called the \emph{sides} of $L$, whose only singular points are ${{n+1}\choose 2}$ nodes, called the \emph{vertices} of $L$. If a curve $C$ passes through all the vertices of $L$ we will say that it is \emph{circumscribed} about $L$ and $L$ is said to be \emph{inscribed} in $C$. If all sides of $L$ are tangent to an irreducible conic $\Gamma$, we will say that $L$ is \emph{circumscribed} about $\Gamma$.

The following lemma is elementary and its proof can be left to the reader (see \cite {Seg}):

\begin{lemma}\label{lem:circ} If $L$ is  a general $(n+1)$--lateral the linear system  of all curves of degree $n$ circumscribed about $L$ has dimension $n$.
\end{lemma}

The following Poncelet type theorem is classical (see \cite [p. 189]{Darb2}, \cite [p. 248]{Darb}, \cite {Seg}), however we will give a new proof of it which is suitable for what we will do later. 

\begin{theorem}\label{thm:darb} Let $C$ be a curve of degree $n$ which is circumscribed about a general $(n+1)$--lateral $L$ that in turn is circumscribed about an irreducible conic $\Gamma$. Then there is a 1--dimensional family of general $(n+1)$--laterals which are inscribed in $C$ and circumscribed about $\Gamma$, so that any point of $C$ is the vertex of such a 
$(n+1)$--lateral.
\end{theorem}

\begin{proof} Identify $\PP^2$ with the symmetric product $\Sym^2(\PP^1)$, and $\Gamma$ with the diagonal. Then any curve $T$ of degree $n$ of the plane not containing $\Gamma$ can be interpreted as a symmetric correspondence of degree $n$ on $\PP^1$ in the following way.  Identify $\PP^1$ with $\Gamma$. Pick a general point $P\in \Gamma$ and consider the tangent $t_P$ to $\Gamma$ at $P$. Then $t_P$ cuts an effective divisor $D_P$ of degree $n$ on $T$.
For a point $Q$ of $D_P$ draw the tangent line to $\Gamma$ different from $t_P$, which cuts $\Gamma$ at a point $Q'$. Doing this for each point of $D_P$ we identify a divisor $T(P)$ of degree $n$ on $\Gamma$ which is associated to $P$ via $T$. 

The  $(n+1)$--lateral $L$ circumscribed about $\Gamma$, determines a divisor $D_L$ of degree $n+1$ on $\Gamma$, i.e., the reduced intersection of $L$ with $\Gamma$. Consider  a linear series $\xi=g^1_{n+1}$ containing $D_L$. Any such a linear series determines a symmetric correspondence of degree $n$ on $\PP^1$, namely the correspondence which maps a general point $P\in \PP^1$ to the unique divisor $T(P)$ of degree $n$ such that $P+T(P)\in \xi$. Hence $\xi$ determines a curve $T_\xi$ of degree $n$ in $\PP^2$ and  different linear series $\xi$ determine different curves $T_\xi$. Morover since $D_L\in \xi$, it is clear that $T_\xi$ is circumscribed about $L$. 

Now the linear series  $\xi=g^1_{n+1}$ containing $D_L$ describe a projective space of dimension $n$. Hence the curves $T_\xi$ associated to these linear series also vary in an irreducible family of dimension $n$. On the other hand the curves of degree $n$ circumscribed about $L$ form a linear system of dimension $n$ by Lemma \ref {lem:circ}. This means that all curves of degree $n$ circumscribed about $L$ are associated to some $\xi=g^1_{n+1}$. This immediately implies the assertion.  \end{proof}

\section{A Poncelet type theorem on symmetric products of general curves}\label {sec:4}

In this section we want to extend Theorem \ref {thm:darb} to curves on the symmetric  product of curves with general moduli. 

Let $C$ be a smooth, irreducible curve of genus $g>0$, and consider $C(2)=\Sym^2(C)$. We will consider the diagonal $D$ of $C(2)$ which is naturally identified with $C$. The class of $D$ in $\Num(C(2))$ is divisible by 2, and we will denote by $\delta$ the class such that $D\equiv 2\delta$. 

On $C(2)$ we have the $1$--dimensional family of \emph{coordinate curves} $C_P=\{P+Q, Q\in C\}$. They are all tangent to $D$ and intersect each other transversely at one point. We will denote by $x$ the class of any curve $C_P$ in $\Num(S)$.  The datum $L$ of $n+1$ distinct such curves is called a \emph{$(n+1)$--lateral} circumscribed about $D$, and $D$ is \emph{inscribed} in $L$. The $(n+1)$--lateral $L$ has $n+1$ \emph{sides}, i.e., the curves $C_P$ forming it, and ${n+1}\choose 2$ \emph{vertices}, i.e., its nodes. 
A curve $T$ on $S$ which contains the vertices of $L$ is said to be \emph{circumscribed} about $L$ and $L$ is said to be \emph{inscribed} in $T$.

The following theorem is classical and well known (see \cite[\S 3]{cilend}):

\begin{theorem}\label{thm:HS} If $C$ has general moduli, then $\Num(C(2))\cong \ZZ\langle x,\delta\rangle$. 
\end{theorem}

This has the following geometric counterpart. Let $T$ be a curve on $S$. Then $T$ determines a symmetric correspondence on $C$. If $T\cdot x=n$, then the correspondence has \emph{degree} $n$, i.e., it maps a point $P\in C$ to an effective divisor $T(P)$ of degree $n$ on $C$. The correspondence is said to have \emph{valence} $\gamma$ if and only if the equivalence class of $\gamma P+T(P)$ does not depend on $P$. This is equivalent to say that $T=(n+\gamma)x-\gamma\delta$ in $\Num(S)$. We will denote by $T_{n,\gamma}$ a correspondence of degree $n$ and valence $\gamma$. Theorem \ref {thm:HS} says that if $C$ has general moduli, then every correspondence on $C$ has valence. 

For our purposes it is useful to consider correspondences $T_{n,1}$. 
Particular correspondences of this type are obtained in the following way. Consider a linear series $\xi=g^1_{n+1}$ on $C$. Then consider the symmetric correspondence $T_\xi$ of degree $n$ and valence 1, which maps a general point $P$ of $C$ to the unique divisor $T(P)$ such that $P+T(P)\in \xi$. 

In analogy with Theorem \ref {thm:darb} we will now prove the:

\begin{theorem}\label{thm:ponc} Let $C$ be a general curve of genus $g>0$. Let $n\geq 2g-2$ be an integer. Consider a $(n+1)$--lateral on $C(2)$ circumscribed about the diagonal $D$. Consider a curve $T=T_{n,1}$  which is circumscribed about $L$. Then  there is a 1--dimensional family of  $(n+1)$--laterals which are inscribed in $L$ and circumscribed about $\Gamma$, so that any point of $T$ is the vertex of such a 
$(n+1)$--lateral.
\end{theorem}

\begin{proof} The proof imitates the one of Theorem \ref {thm:darb}. Indeed, the $(n+1)$--lateral $L$ determines a divisor $D_L$ of degree $n+1\geq 2g-1$ on $D$, namely the reduced divisor cut out by $L$ on $D$. This divisor determines the complete linear series $|D_L|=g_{n+1}^{n+1-g}$ on $D=C$. The linear series $\xi=g_{n+1}^{1}$ contained in $|D_L|$ and containing $D_L$ are parameterized by a projective space of dimension $n-g$. Each such series gives rise to a curve $T_\xi=T_{n,1}$ which is circumscribed about $L$, and it clearly verifies the assertion. So to prove the theorem it suffices to show, as in the proof of Theorem \ref {thm:darb}, that the curves of type $T_{n,1}$  that are circumscribed about $L$ form an irreducible system of dimension $n-g$. 

To see this, we consider Franchetta's degeneration of the symmetric product $C(2)$ when $C$ tends to a rational curve with $g$ general nodes, as explained in \cite {CK}. The limit of the surface $C(2)$ is a rational surface which can be represented on the plane in such a way that the system of curves $T_{n,1}$ circumscribed to $L$ goes to a system of curves of degree $n$, circumscribed to a general $(n+1)$--lateral $L_0$ circumscribed about an irreducible conic $\Gamma$, which also pass through $g$ general points $P_1,\ldots, P_g$ of the plane and cut on the  
two tangent lines $r_{i,1}, r_{i,2}$ to $\Gamma$ through each of the points $P_i$ pairs of points  corresponding to each other under the projectivity $\omega_i: r_{i,1}\to r_{i,2}$ induced by intersections with the tangent lines to $\Gamma$, for $i=1,\ldots, g$.

Now, by Lemma \ref {lem:circ}, the linear system $\C$ of curves of degree $n$ circumscribed to $L_0$ has dimension $n$, and, by Theorem  \ref {thm:darb},  each point of any curve $C$ in $\C$ is the vertex of such a $(n+1)$--lateral inscribed in $C$ and circumscribed to $\Gamma$. So, if we impose to the curves of $\C$ to contain $P_1,\ldots, P_g$, which is $g$ independent conditions because $P_1,\ldots, P_g$ are general points, we obtain a linear system of dimension $n-g$ of curves which are limiting curves of $T_{n,1}$  on $S$. This proves the theorem. \end{proof}

\printindex

\end{document}